\documentclass[a4paper,reqno, 10pt]{amsart}

\author[ W. Clark, A. Bloch, L. Colombo]{William Clark, Anthony Bloch, Leonardo Colombo}
\usepackage{amsfonts}
\usepackage{graphicx}
\usepackage{amsmath, amssymb}
\usepackage{mathrsfs}
\usepackage{amsthm}
\usepackage{multirow}
\usepackage{mathtools}
\usepackage{ marvosym }
\usepackage{tikz-cd}
\usepackage{tikz}
\usepackage{subcaption}

\newtheorem{theorem}{Theorem}[section]
\newtheorem{lemma}[theorem]{Lemma}
\newtheorem{proposition}[theorem]{Proposition}
\newtheorem{corollary}[theorem]{Corollary}
\newtheorem{example}{Example}[section]

\newtheorem{remark}[theorem]{Remark}
\usepackage[mathscr]{euscript}
\def\spoke#1#2{
\begin{scope}[shift={#1}, rotate=#2]
    \draw[thick] (0,-.15) -- (0,-1);
\end{scope}
}

\theoremstyle{definition}
\newtheorem{definition}[theorem]{Definition}

\theoremstyle{proposition}
\theoremstyle{conjecture}
\theoremstyle{claim}

\usepackage{verbatim}
\usepackage{graphicx}
\usepackage{float}
\usepackage{color}
\usepackage{fancyhdr}
\usepackage{authblk}
\usepackage{tikz-cd}

\definecolor{mygreen}{rgb}{0,0.6,0}
\definecolor{mygray}{rgb}{0.5,0.5,0.5}
\definecolor{mymauve}{rgb}{0.58,0,0.82}
\definecolor{light-gray}{gray}{0.95}

\newcounter{casenum}
\newenvironment{caseof}{\setcounter{casenum}{1}}{\vskip.5\baselineskip}
\newcommand{\case}[2]{\vskip.5\baselineskip\par\noindent {\bfseries Case \arabic{casenum}:} #1\\#2\addtocounter{casenum}{1}}

\newcommand{\fix}{\text{fix}}
\newcommand{\N}[1]{\mathcal{N}_{f}(#1)}

\newcommand{\X}{\mathcal{X}}

\title{A Poincar\'e-Bendixson Theorem for Hybrid Systems}

\address{William Clark: Department of Mathematics, University of Michigan. 530 Church Street, Ann Arbor, MI, USA.} \email{wiclark@umich.edu}
\address{Anthony Bloch: Department of Mathematics, University of Michigan. 530 Church Street, Ann Arbor, MI, USA.} \email{abloch@umich.edu}
\address{Leonardo Colombo: Department of Automatic Control, School of Electrical Engineering and Computer Sciences, KTH Royal Institute of Technology, SE-100  44,  Stockholm,  Sweden.} \email{colombo2@kth.se}

\begin{document}

\thanks{A. Bloch was supported by NSF grant DMS-1613819 and AFOSR. W. Clark was supported by NSF grant DMS-1613819 and L. Colombo was supported by MINECO (Spain) grant MTM2016-76072.}
\thanks{We thank Dr. Jessy Grizzle for valuable feedback and discussion on this paper.}


\maketitle
\begin{abstract} The Poincar\'e-Bendixson theorem plays an important role in the study of
the qualitative behavior of dynamical
systems on the plane; it describes the structure of limit sets in such
systems. We prove  a version of the Poincar\'e-Bendixson Theorem for two dimensional hybrid dynamical systems and describe a method for computing the derivative of the Poincar\'e return map, a useful object for the  stability analysis of hybrid systems. We also prove a Poincar\'e-Bendixson Theorem for a class of one dimensional hybrid dynamical systems.
\end{abstract}
\section{Introduction}

H. Poincar\'e considered the problem of characterizing the structure of limit sets of trajectories of analytic vector fields on the plane in 1886 \cite{poincare}.
I. Bendixson improved the solution proposed by Poincar\'e in 1903 by solving the problem under the weaker hypothesis of $C^1$ vector fields \cite{bendixson}. 
Since then, the investigation of the asymptotic behavior of dynamical systems has been essential to understanding their behavior. The theory of Poincar\'e-Bendixson studies so-called \textit{limit sets}. The classic version of the Poincar\'e-Bendixson Theorem states that if a trajectory is bounded and its limit set does not contain any fixed points, then the limt set is a periodic orbit  \cite{perko1991differential}. Therefore, the problem of determining the existence of limit cycles in planar continuous dynamics is well understood.

Hybrid systems \cite{teelSurvey} are non-smooth dynamical systems which exhibit a combination of smooth and discrete dynamics, where the flow evolves continuously on a state space, and a discrete transition occurs when the flow intersects a co-dimension one hypersurface.
Due to many engineering applications, such as dynamical walking of bipedal robots \cite{collins}, \cite{grizzle}, \cite{grizzle2}, there has been an increased interest in recent years in studying the existance and stability of limit cycles in hybrid systems \cite{Lou2015OnRS}, \cite{Lou2015ResultsOS}, \cite{Lou2017ExistenceOH}.
There have also been several attempts at building a foundational qualitative theory for hybrid systems (see \cite{Matveev:2000:QTH:555969}, \cite{SIMIC2002197} and references therein)
where early versions of the Poincar\'e-Bendixson Theorem were developed. However, many fundamental questions solved for continuous-time systems still remain open for hybrid systems. The results in \cite{Matveev:2000:QTH:555969} are restricted to the situation of constant vector fields while in \cite{SIMIC2002197}, the authors considered a
particular class of systems with much stricter assumptions to ensure the existence of periodic orbits.

The Poincar\'e map, or first return map, is a method of studying the stability of limit cycles by reducing the dimension of the dynamics for a continuous-time dynamical system by one, and considering this as a discrete-time system \cite{guckenheimer2002nonlinear}, \cite{perko1991differential}, \cite{strogatz2014nonlinear}. 
The map is constructed as follows: if $\gamma$ is a periodic orbit that intersects a hypersurface (the Poincar\'e section) transversely at a point $x_0$, then for a point $x$ near $x_0$, the solution through $x$ will cross the hypersurface again at a point $P(x)$ near $x_0$. The mapping $x\mapsto P(x)$ is called the Poincar\'e map. In practice, it is impossible to find the Poincar\'e map analytically and in closed form due to the fact that it requires the solution to the differential equation.
The extension of the continuous-time Poincar\'e map for mechanical systems with impulsive effects was considered in \cite{grizzle}. The hybrid Poincar\'e map is important in applications as it is used to ensure the existence and stability properties of periodic locomotion gaits \cite{siamreview}. 

In this work, we study the problem of existence and stability of periodic orbits for hybrid dynamical systems, in addition to other results concerning the qualitative behavior of these systems.
Among these results, we present a version of the Poincar\'e-Bendixson Theorem for two dimensional hybrid systems under weaker conditions than the ones considered in \cite{Matveev:2000:QTH:555969}, \cite{SIMIC2002197}. We also derive an analytical method to compute the derivative of the hybrid Poincar\'e map to characterize the stability of periodic orbits. We apply our results to find a region in parameter space where one can ensure the existence of limit cycles for the rimless wheel, a popular system in locomotion research used to study essential properties of walking robots. We additionally prove a Poincar\'e-Bendixson Theorem for a general class of one dimensional hybrid dynamical systems.

The paper is organized as follows: Section \ref{sec:hybrid} introduces the formulation of hybrid dynamical systems as well as hybrid $\omega$-limit sets. Section \ref{sec:limit} contains two properties of the hybrid $\omega$-limit set that are needed to properly formulate our hybrid Poincar\'e-Bendixson theorem.  Section \ref{sec:poincare} introduces the hybrid Poincar\'e map and uses this to prove our hybrid Poincar\'e-Bendixson Theorem (Theorem \ref{th:HPB}). Section \ref{sec:stability} offers an analytic way to compute the derivative of the hybrid Poincar\'e map, then uses this result  to study the stability of planar hybrid limit cycles. Section \ref{sec:1dim} contains a version of the Poincar\'e-Bendixson theorem for a general class of one dimensional hybrid dynamical systems. The paper ends with an application of the main theorem to find conditions for stability of periodic walking of the rimless wheel. 
Appendix \ref{appendix} contains some analogue results of this work for time-continuous flows.

\section{Hybrid Dynamical Systems}\label{sec:hybrid}
Hybrid dynamical systems (HDS) are dynamical systems characterized by their mixed behavior of continuous and discrete dynamics where the transition is determined by the time when the continuous flow switches from the ambient space to a co-dimensional one submanifold. This class of dynamical systems is given by an $4$-tuple, $(\X,S,f,\Delta)$. The pair ($\X$,$f$) describes the continuous dynamics as

\begin{equation*}
\dot{x}(t) = f(x(t))
\end{equation*} where $\X$ is a smooth manifold and $f$ a $C^1$ vector field on $\X$ with flow $\varphi_t:\X\rightarrow\X$. Additionally, ($S$,$\Delta$) describes the discrete dynamics as 
$x^{+}=\Delta(x^{-})$ where $S\subset\X$ is a smooth submanifold of co-dimension one called the \textit{impact surface}.

The hybrid dynamical system describing the combination of both dynamics is given by
\begin{equation}\label{sigma}\Sigma: \begin{cases}
\dot{x} = f(x),& x\not\in S\\
x^+ = \Delta(x^-),& x^-\in S.
\end{cases}\end{equation}
The flow of the hybrid dynamical system \eqref{sigma} is denoted by $\varphi_t^H$. This may cause a little confusion around the break points, that is, where $\varphi_{t_0}(x)\in S$. Then, is $\varphi_{t_0}^H(x) = \varphi_{t_0}(x)$ or $\varphi_{t_0}^H(x) = \Delta(\varphi_{t_0}(x))$? That is, at the time of impact with the submanifold $S$, is the state $x^-$ or $x^+$?
We will take the second value. i.e. $\varphi_{t_0}^H(x)=x^+$.
However, this means that the orbits will (in general) not be closed. 


\begin{definition}
The (forward) \textit{orbit} and the \textit{$\omega$-limit set} for the hybrid flow $\varphi_t^H(x)$ are given by
\begin{equation}\begin{split}
&o^+_H(x) := \left\{ \varphi_t^H(x) : t\in\mathbb{R}^+\right\}\\
&\omega_H(x) := \left\{ y\in\X : \exists t_n\rightarrow\infty ~ s.t. \lim_{n\rightarrow\infty} \varphi_{t_n}^H(x) = y\right\}
\end{split}
\end{equation}

%

Additionally, we define the set $\fix(f)$ of fixed points for a function $f$ and the covering set of fixed points $\N{\varepsilon}$ as
\begin{equation}\label{eq:fixed_sets}
\begin{split}
\fix(f) &:= \left\{ y\in \X : f(y)=0\right\} \\
\N{\varepsilon} &:= \bigcup_{x\in\fix(f)} \; \mathcal{B}_\varepsilon(x)
\end{split}
\end{equation}
where $\mathcal{B}_\varepsilon (x)$ is the open ball of radius $\varepsilon$ around the point $x$. 
\end{definition}

The main problem studied in this work is that of proving an analogue of the Poincar\'e-Bendixson Theorem for continuous-time planar dynamical systems for a suitable class of planar HDS. For this purpose, we will consider a slightly more general HDS form than the one studied in \cite{grizzle}.


\begin{definition}\label{def:smooth_hybrid}
	A 4-tuple, $(\X,S,f,\Delta)$, forms a hybrid dynamical system if
	\begin{enumerate}
		\item[(H.1)] $\X\subset\mathbb{R}^n$ is open and connected.
		\item[(H.2)] $f:\X\rightarrow\mathbb{R}^n$ is $C^1$.
		\item[(H.3)] $H:\X\rightarrow\mathbb{R}$ is $C^1$.
		\item[(H.4)] $S:=H^{-1}(0)$ is non-empty and for all $x\in S$, $\displaystyle{\frac{\partial H}{\partial x}\ne0}$ (so $S$ is $C^1$ and has co-dimension 1).
		\item[(H.5)] $\Delta:S\rightarrow \X$ is $C^1$.
		\item[(H.6)] $\overline{\Delta(S)}\cap S\subset \fix(f)$ and they intersect transversely.
	\end{enumerate}
\end{definition}
Note that assumptions (H.1) and (H.2) are required for the continuous flow to exist and be unique. (H.3) and (H.4) make the impact surface well defined, according to \cite{grizzle}. The assumption (H.5) is included because without it, the $\omega_H$-limit set is not (in general) invariant under the flow. The last assumption, (H.6), is to rule out the Zeno phenomenon away from fixed points. (A flow experiences a \textit{Zeno state} (\cite{teelSurvey},\cite{SIMIC2002197}) if the flow $\varphi_t^H$ intersects $S$ infinitely often in a finite amount of time.) Assumption (H.6) is slightly weaker than as presented in \cite{grizzle}, where it is assumed that $\overline{\Delta(S)}\cap S=\emptyset$ (or equivalently, the set of impact times is closed and discrete). 

\begin{remark}
	Dropping hypothesis (H.5), $\omega_H(x)$ is not always invariant under the flow $\varphi_t^H$. That is, if $p\in \omega_H(x)$, then $o_H^+(p)\not\subset \omega_H(x)$.
	The following example shows this situation.
\end{remark}
\begin{example}
Consider the following hybrid system: Let the state-space be $\X=[0,1]\times\mathbb{R}\subset\mathbb{R}^2$ and the continuous flow  be determined by
$\dot{x}=1$ and $\dot{y} = -y^2$. Let the impact surface be $S = \{ (1,y):y\in\mathbb{R}\}$ and the impact map be given by
$$\Delta(1,y) = \begin{cases}
(0,y) & y > 0\\
(0,y-1) & y \leq 0.
\end{cases}$$
The $\omega_H$-limit set of the starting point $(0,1)$ is the interval $[0,1]\times\{0\}$ which is clearly not invariant under the flow of the system because the impact moves the flow away from $\omega_H(0,1)$.
\end{example}


\section{Properties of Hybrid Limit Sets}\label{sec:limit}
In this section we study two relevant properties of hybrid limit cycles. First, we study sufficient conditions for which the $\omega_H$-limit set is nonempty and compact in analogy with Theorem \ref{th:perko} in the
Appendix, but for hybrid flows. The result is used to show that with assumption (H.5) the $\omega_H$-limit set is indeed invariant.
\begin{proposition}\label{th:closed}
	The $\omega_H$-limit set of a trajectory $o_H^+(x)$ is a closed set. Additionally, if $R$ is compact and forward invariant set, then $\omega_H(x)$ is nonempty and compact for  $x\in R$.
\end{proposition}

\begin{proof}
	The proof follows the same arguments as in Theorem \ref{th:perko} (see \cite{perko1991differential} pp. 193).
	
	 First, let us  prove that $\omega_H(x)$ is closed. Let $\{p_n\}_{n\in\mathbb{N}}$ be a sequence in $\omega_H(x)$, such that $p_n\rightarrow p$ when $n\to\infty$. We want to show that $p\in\omega_H(x)$. Since $p_n\in\omega_H(x)$, there exists a sequence of times, $\{t_k^{(n)}\}$, such that $\varphi_{t_k^{(n)}}^H(x)\rightarrow p_n$ when $t_k^{(n)}\rightarrow\infty$. Without loss of generality, consider $t_k^{(n+1)}>t_k^{(n)}$. Then, for all $n\geq 2$ there exists $K_n>K_{n-1}$ such that for all $k\geq K_n$
	$$\left| \varphi_{t_k^{(n)}}^H(x)-p_n\right| < \frac{1}{n}.$$
	Choose a sequence of times $t_n = t_{K_n}^{(n)}$. Then, by the triangle inequality, as $t_n\rightarrow\infty$ we obtain that $\varphi_{t_n}^H(x)$ converges to $p$, that is,
	$$\left| \varphi_{t_n}^H(x) - p\right| \leq \left| \varphi_{t_n}^H(x)-p_n\right| + \left| p_n-p\right| \leq
	\frac{1}{n} + \left| p_n - p \right| \rightarrow 0 {\hbox{ when } n\to\infty}.$$
	For the second part, we have that $\omega_H(x)\subset R$, so it is compact {since it is a closed subset of a compact set}. To show that it is nonempty, we point out that the sequence $\{ \varphi_n^H(x)\}_{n\in\mathbb{N}}$ is in a compact set so by Bolzano-Weierstrass {Theorem}, there exists a convergent subsequence.
\end{proof}
\begin{remark}Note that $\omega_H(x)$ is closed but $o_H^+(x)$ is not.\end{remark}
\begin{proposition}
	$\omega_H(x)$ is invariant under the flow, {$\varphi_t^H$}, i.e., if $x\in\X$, for all $p\in \omega_H(x)$, $o_H^+(p)\subset\omega_H(x)$.
\end{proposition}
\begin{proof}


	The proof of the analogous theorem {for continuous time systems given in} \cite{perko1991differential} {(see Theorem $2$, pp.194)} depends on the trajectories changing continuously based on initial conditions. This {argument} clearly does not work {for hybrid flows}, so we need to modify what continuous means. We do this by identifying points as being close if they are on opposite sides of the jump. 
	
	Define an equivalence relation on $\X$ by $x\sim y$ if $x=y$ or if $x=\Delta(y)$ for $y\in S$. Re-topologize $\X$ by defining open balls via
	$$\tilde{\mathcal{B}}_\varepsilon(x) = \bigcup_{y\in [x]} \mathcal{B}_\varepsilon(y).$$
	Then, under this topology the flow, $\varphi_t^H$ is continuous. Additionally, we now have continuous dependence on initial conditions. If the flow takes us away from the impact surface, we get continuous dependence under normal continuous flows. If it takes us to the impact surface, we are still continuous because the impact is continuous.\\
	\indent {Consider} $q\in o_H^+(p)$. Let $t_0$ {be the time} such that $q = \varphi_{t_0}^H(p)$. Additionally, {let $\{t_n\}_{n\in\mathbb{N}}$ be a sequence of times  such that} $\varphi_{t_n}^H(x)\rightarrow p$ {as $t_n\to\infty$}.
	By the semi-group property of flows and the continuity of the flow with respect to initial conditions, we get
	$$\varphi_{t_0+t_n}^H(x) = \varphi_{t_0}^H\circ \varphi_{t_n}^H(x)
	\rightarrow \varphi_{t_0}^H(p) = q.	$$
\end{proof}

\section{Poincar\'e-Bendixson theorem for $2$ dimensional hybrid dynamical systems}\label{sec:poincare}

\subsection{Preliminary result for discrete dynamical systems}

\begin{definition}
Let $S$ be a smooth manifold and $P:S\rightarrow S$ be $C^1$. 
The \textit{discrete flow} is defined as
\begin{equation}\label{eq:disc}
x_{n+1}=P(x_n).
\end{equation}
The \textit{discrete} $\omega_d$-\textit{limit set} is defined as
\begin{equation}\label{eq:discrete_omega}
\omega_d(x):=\left\{ y\in S : \exists N_n\rightarrow\infty~s.t.~ \lim_{n\rightarrow\infty}P^{N_n}(x)=y \right\}.
\end{equation}
\end{definition}

The set $\omega_d(x)$ satisfies the following property
\begin{lemma}\label{le:single}
	Let $P:[a,b]\rightarrow[a,b]$ be $C^1$ and injective. Then for all $x\in[a,b]$, $\omega_d(x)$ is either a single point or two points. i.e. all trajectories approach a periodic orbit.
\end{lemma}
\begin{proof}
	First, because $P$ is invertible on its image and differentiable, $P'>0$ or $P'<0$ on the entire interval. Without loss of generality, assume that it is increasing (by examining $P^2$ if $P'<0$).  Next, since we are condsidering the $\omega_d-$limit set, we can take an iterate of $P$. This makes the system into $P:[c,d]\tilde{\rightarrow}[c,d]$ where $c=P(a)$ and $d=P(b)$. Additionally, since $P$ is  a bijection and is continuous and increasing, we must have $P(c)=c$ and $P(d)=d$.
	
	Define the closed set $F:=\{x:P(x)=x\}$. Then, if $x\in F$ we are done. So assume that $x\not\in F$. Let $a_1\in F$ be the maximal element of $F$ less than $x$ and $a_2\in F$ be the minimal element greater than $x$. Also, call the invariant interval $I=(a_1,a_2)$. Since $P-Id$ does not have a root on $I$ and is continuous, we have two possibilities for all $y\in I$: either $P(y)>y$ or $P(y)<y$.
	
	If $P(y)<y$ for all $y\in I$, then the sequence $P(x),P^2(x),\ldots$ is  monotone decreasing and thus is convergent. Likewise, if $P(y)>y$, $P^n(x)$ is a monotone increasing sequence. Thus, $P^n(x)$ always converges and its limit set must be a single point, or two points if we're dealing with $P^2$.
\end{proof}


Thus, if $P$ is injective, the $\omega_d$-limit set is either a periodic orbit or a fixed point. 

\subsection{Existence of hybrid Poincar\'e map}
To study periodic orbits it is useful to take a Poincar\'e section, of which it would seem natural to take $S$ as the section \cite{grizzle}. The problem is that for a given $x\in S$, it is not guaranteed that $P'(x)$ exists. The next theorem addresses this problem.
\begin{theorem}\label{thm:smooth}
	Let $x_0\in S\setminus\fix(f)$ be such that there exists a time, $T_0>0$, where $\varphi_{T_0}(x_0)\in S$. Additionally, assume that the flow intersects the impact surface transversely at $\varphi_{T_0}(x_0)$. Then there exists an $\varepsilon>0$ and a $C^1$ function $\tau:\mathcal{B}_{\varepsilon}(x_0)\cap S\rightarrow \mathbb{R}^+$ such that for all $y\in \mathcal{B}_{\varepsilon}(x_0)\cap S$, $\varphi_{\tau(y)}(y)\in S$. 
\end{theorem}
\begin{proof}
	Define the function $F:(0,+\infty)\times S\rightarrow \X$ by $F(t,x) = H(\varphi_t(\Delta(x)))$. It follows from Theorem 1 in Section 2.5 in \cite{perko1991differential} that $(t,x)\mapsto \varphi_t(x)$ is $C^1(\mathbb{R}\times \X)$. Combining this with the fact that both $H$ and $\Delta$ are $C^1$ functions, we get that their compositions are. Since $F\in C^1( \mathbb{R}^+ \times S)$, we can use the implicit function theorem. At our point $x_0\in S$, we know that the orbit enters the set $S$ at some minimal future time, $T_0$. This gives $F(T_0,x_0)=0$.
	Differentiating $F$ with respect to time yields:
	$$\frac{\partial F}{\partial t} (T_0,x_0) = \left. \frac{\partial H}{\partial y} \right|_{y=\varphi_{T_0}(\Delta(x_0))\in S} \cdot f(\varphi_t(\Delta(x_0))) \ne 0.$$
	The first factor is nonzero because of assumption (H.4) and the second is nonzero because we are away from a fixed point. Their inner product is nonzero because of the transversality condition. This lets us use the implicit function theorem (cf., e.g., Theorem 9.28 in \cite{rudin1976principles}), to show that there exists a neighborhood of $\mathcal{B}_{\varepsilon}(x_0)$ of $x_0$ and a $C^1$ function $\tau$ with all the desired properties.
\end{proof}

\subsection{Poincar\'e-Bendixson theorem for planar HDSs}

With the existence of a smooth Poincar\'e map, we can now prove the Poincar\'e-Bendixson theorem for planar HDSs.
\begin{theorem}\label{th:HPB}
	Assume the conditions (H.1)-(H.6). Additionally, assume
	\begin{enumerate}
		\item[(C.1)] $\X\subset \mathbb{R}^2$.
		\item[(C.2)] $\Delta:S\rightarrow \X$ is injective.
		\item[(C.3)] There exists a forward, invariant, compact set $F\subset \X$ with $F\cap \fix(f)=\emptyset$.
		\item[(C.4)] $F\cap S$ is diffeomorphic to an interval.
		\item[(C.5)] The vector field, $f$, is transverse to both $F\cap S$ and $\Delta(F\cap S)$.
	\end{enumerate}
	Then, if $x_0\in F$, $\omega_H(x_0)$ is a periodic orbit. Moreover, $\omega_H(x_0)$ intersects the impact surface, $S$, at most twice.
\end{theorem}
\begin{proof}
	For the entirety of this proof, we will redefine $S$ by $F\cap S$. That is, $S$ is diffeomorphic to an interval.
	Consider the Poincar\'e return map, $P:x\mapsto \varphi_{\tau(\Delta(x))}(\Delta(x))$. The domain of this differentiable function is an open subset of $S$ (by Theorem \ref{thm:smooth}). Call this set $S^1$. i.e.
	$$S^1 := \{ x\in S:~\exists t>0~with~\varphi_t(\Delta(x))\in S\}$$
	Specifically, we want to look at all points of $S$ that return back to $S$ infinitely often. Call this set $S^{\infty}$. We can define $S^{\infty}$ recursively as such.
	\begin{equation*}
	\begin{split}
	S^{n+1} := \{ x\in S^n &:~\exists t>0 ~with ~\varphi_t(\Delta(x))\in S^n\}\\
	S^\infty := &\bigcap_{n=1}^{\infty}S^n
	\end{split}
	\end{equation*}
	We have two cases for $x_0$: either $o_H^+(x_0)$ hits $S^\infty$ (and thus the impact surface infinitely often), or $o_H^+(x_0)$ avoids $S^\infty$.
	\begin{caseof}
		\case{$o_H^+(x_0)$ misses $S^\infty$.}{
			In this case, there exists a time large enough where the flows completely stops being hybrid. In this setting, we can invoke the normal Poincar\'e-Bendixson theorem for continuous systems. See Theorem 1 in chapter 3.7 in \cite{perko1991differential}.
		}
		\case{$o_H^+(x_0)$ hits $S^\infty$.}{
			Because $o_H^+(x_0)\cap S^\infty\ne\emptyset$, we know that $S^\infty\ne\emptyset$.
			We wish to show that $S^\infty$ is either an interval or a point. 
			This will let us use Lemma \ref{le:single} to show that $P:S^\infty\rightarrow S^\infty$ converges to a limit cycle.
			
			We will begin by showing that for each $n\geq 0$, $S^n$ is an interval. By assumption (C.4), $S^0$ is an interval. We will continue by induction. Assume that $S^n$ is an interval and we wish to prove that $S^{n+1}$ is also an interval.
			
			Since $S^n$ is diffeomorphic to an interval, let $g_n:S^n\rightarrow [a_n,b_n]$ be a diffeomorphism.
			Define the points $a_{n+1}$ and $b_{n+1}$ as follows:
			\begin{equation}
			\begin{split}
			a_{n+1} &:= \min \left\{ 
			x\in [a_n,b_n] : o^+(\Delta( g_n^{-1}(x)))\cap S^n\ne\emptyset
			\right\} \\
			b_{n+1} &:= \max \left\{ 
			x\in [a_n,b_n] : o^+(\Delta( g_n^{-1}(x)))\cap S^n\ne\emptyset
			\right\}.
			\end{split}
			\end{equation}
			We claim that $S^{n+1}$ is diffeomorphic to $[a_{n+1},b_{n+1}]$. Denote the points on the curve $S^n$ by  $A:=g_n^{-1}(a_{n+1})$ and $B:= g_n^{-1}(b_{n+1})$. The claim can then be verified by constructing the set $\mathcal{B}$ under the continuous dynamics as the region bounded by the four curves: $\Delta( [A,B])$, $S^n$, $o^+(A)$, and $o^+(B)$. Using assumption (C.5), we know that for every initial condition on $\Delta([A,B])$, the flow will eventually hit the set $S^n$. Thus we can be apply the Rectification Theorem \cite{arnolʹd1978ordinary}, to straighten out the flow.
			
%
		}
	\end{caseof}
\end{proof}
Conditions (C.2), (C.4), and (C.5) are unfortunate restrictions, however, they are necessary. If (C.4) is dropped, the flow can end up looking like a Kronocker flow; see Example \ref{ex:c4}. If (C.2) or (C.5) are dropped, mixing can be added to the system and chaos can occur; see Example \ref{ex:c5}
\begin{example}\label{ex:c4}
	Let $S=x^2+y^2-4$. Thus the impact surface is the circle of radius 2 centered about the origin. Let the impact map be given by $\Delta(x,y)=(x/2,y/2)$, so the image of the map is the unit circle. Lastly, define the vector field to be (in polar coordinates) $\dot{r}=\dot{\theta}=1$. Then, $F=\{ 1\leq|r|\leq2\}$ is a compact invariant set. But for all $x\in F$, $\omega_H(x)=F$.
\end{example}
\begin{example}\label{ex:c5}
	Let $S = \{ x=2\}$ and define $\Delta$ as $\Delta(2,y) = (y,4y(1-y))$. Then, if the flow is $\dot{x}=1$, $\dot{y}=0$, the first return map becomes the Logistic map which leads to chaos (see \cite{bookHSD} pp. 344 for more details).
\end{example}
\section{Stability of Periodic Orbits}\label{sec:stability}

Given that we have now a method to determine the existence of periodic orbits, we would like to be able to determine their stability. As such, we would like to be able to compute the derivative of $P$ and get a result analogous to Theorem \ref{th:AA} for hybrid dynamical systems.
There are a couple of {differences we are faced with in the hybrid approach as opposed to the continuous-time situation}. First, we do not get to choose $\Sigma$ to be normal to the flow as in \cite{perko1991differential}; we are stuck with $\Sigma=S$. Second, we are no longer dealing with a closed orbit and we have to take the geometry of the impact into consideration. We first look at a helpful {result} about the continuous flow {we will use in Theorem \ref{th:stability}}.

\begin{lemma}[\cite{perko1991differential}, p. 86]\label{perkodiv}
	Let $\varphi_t(x_0)$ be the flow of $\varphi_t:\X\rightarrow\X, \frac{d}{dt}\varphi_t(x) = f(\varphi_t(x))$ with initial condition $x_0$. Then,
	\begin{equation}\label{eq:det_of_partials}
	\det \left.\frac{\partial}{\partial x} \varphi_t(x)\right|_{x=x_0}=
	\exp \left(\int_0^t \, \nabla \cdot f(\varphi_s(x_0))\, ds\right).
	\end{equation}
\end{lemma}
To understand the stability of our orbit, we want to look at the hybrid Poincar\'e return map, $P:S^1\rightarrow S$. As in Theorem \ref{thm:smooth}, let $\tau:\Delta(S^1)\rightarrow\mathbb{R}$ be the time required to return to the impact surface. Then, if we denote $y:=\Delta(x)$, we can write $P$ as
\begin{equation}\label{eq:poincare_expanded}
P(x) = \varphi_{\tau(y)}(y) = \int_{0}^{\tau(y)} \, f\left(\varphi_s(y)\right)\, ds.
\end{equation}
\begin{theorem}\label{th:stability}
	Assume that we have a hybrid periodic orbit that intersects $S$ once. Suppose that $x\in S$ and $y=\Delta(x)$. Additionally, let $\theta$ be the angle $f(x)$ makes with the tangent of $S$ at $x$ and $\alpha$ be the angle of $f(y)$ with $\Delta(S)$. Assume that $\theta$ and $\alpha$ are not integer multiple of $\pi$. If we denote the continuous flow that connects $y$ to $x$ by $\gamma(t)$ and suppose that it takes time $T$ to complete the loop, the derivative of the Poincar\'e map is
	\begin{equation}\label{eq:finally}
	\left. P'(x) \right. = \Delta'(x)\cdot\frac{ \lVert f(y)\rVert}{\lVert f(x)\rVert}
	\frac{\sin\alpha}{\sin\theta} \cdot\exp \left(\int_0^T  \nabla\cdot f(\gamma(t))\, dt\right).
	\end{equation}
\end{theorem}

\begin{proof}

	\begin{figure}[H]\label{deriv_pic}
		\includegraphics[scale=0.4]{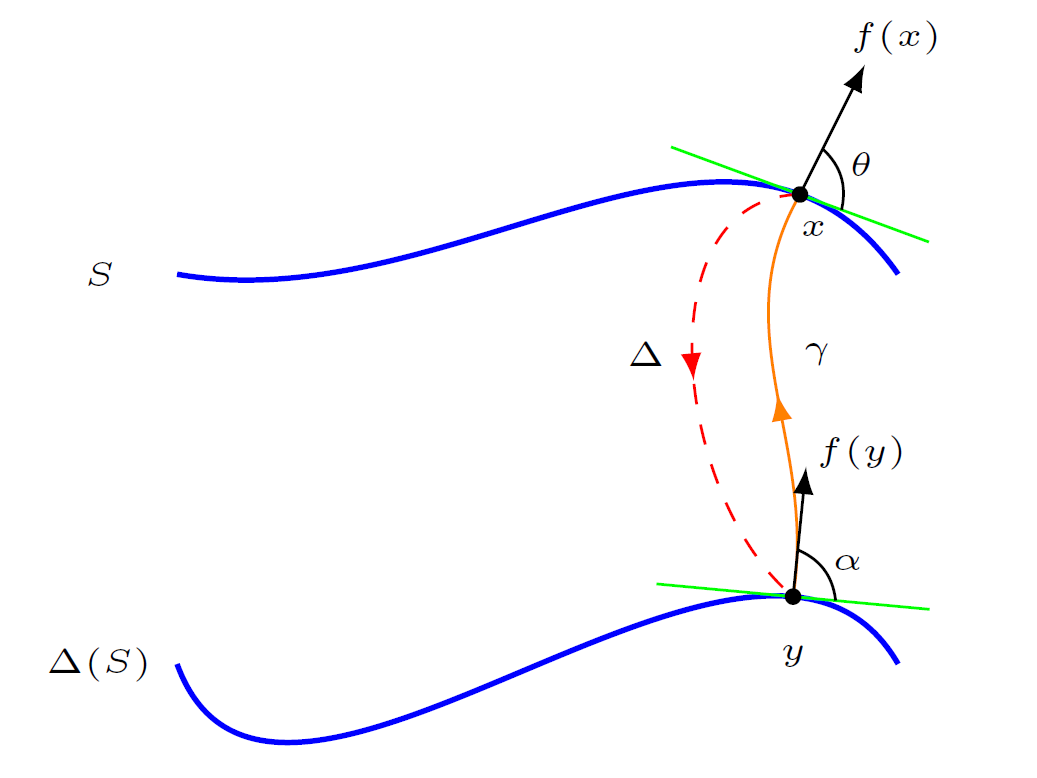}
		\caption{The orbit of the periodic orbit for the system given by Theorem \ref{th:stability}.}
	\end{figure}
	To differentiate $P$, let's first look at the continuous part (that is, starting at $y_0=\Delta(x_0)$).
	Let $n$ be the unit normal vector to $\Delta(S)$ at $y$ and let $p$ be the unit tangent vector. To make things reasonable, we want $\langle f(y_0),n_0\rangle\ne0$.
	\begin{equation}\label{eq:partials}
	\begin{split}
	\left.\frac{\partial}{\partial p} \varphi_{\tau(y)}(y)\right|_{y=y_0} 
	&= \int_0^{\tau(y)}\, \left.\frac{\partial}{\partial y} f\left( \varphi_s(y)\right)\right|_{y=y_0} \, ds \cdot \frac{\partial y}{\partial p} + 
	\frac{\partial}{\partial t} (\varphi_{\tau(y_0)}(y_0)) \cdot
	\frac{\partial t}{\partial p}\\
	&= F(y_0)\cdot \delta y + G(y_0)\cdot \delta t
	\end{split}
	\end{equation}
	Now call the flow $\varphi_t(y)=:\gamma(t)$, the time $T=\tau(y)$, and recall that the final point is $\varphi_{\tau(y)}(y)=x$. Then, $G(y) = f(x)$ and $\delta y$ is the unit vector $p$ rooted at $y_0$. We need to figure out what $\delta t$ and $F(y)$ are. By Lemma \ref{perkodiv}, we know the determinant of $F(y)$.
	\begin{equation}\label{eq:Fy}
	\det\left( F(y) \right) = \exp \left(\int_0^T \, \nabla \cdot f(\gamma(t)) \, dt\right)
	\end{equation}
	To find $F$ (in the direction of $\delta y$), we note that we know the derivative in the direction of the flow: $F(y)\cdot f(y) = f(x)$. Knowing the determinant and this direction, we can attempt to find $F(y)$ in the direction of $\delta y$. We first differentiate $H$ from (H.3) along $S$, which is zero because $S$ is a level set of $H$.
	\begin{equation}
	0=\left.\frac{\partial}{\partial p}H(\varphi_{\tau(y)}(y))\right|_{y=y_0} = \left.\left.\frac{\partial}{\partial x}H(x)\right|_{x=x_0} \cdot
	\left( F(y)\cdot \delta y + f(x)\cdot \delta t \right)\right|_{y=y_0}
	\end{equation}
	This tells us that $F(y)\cdot\delta y + f(x)\cdot\delta t$ lies on the tangent to $S$ at $x$.
	\begin{figure}[H]
		\includegraphics[scale=0.27]{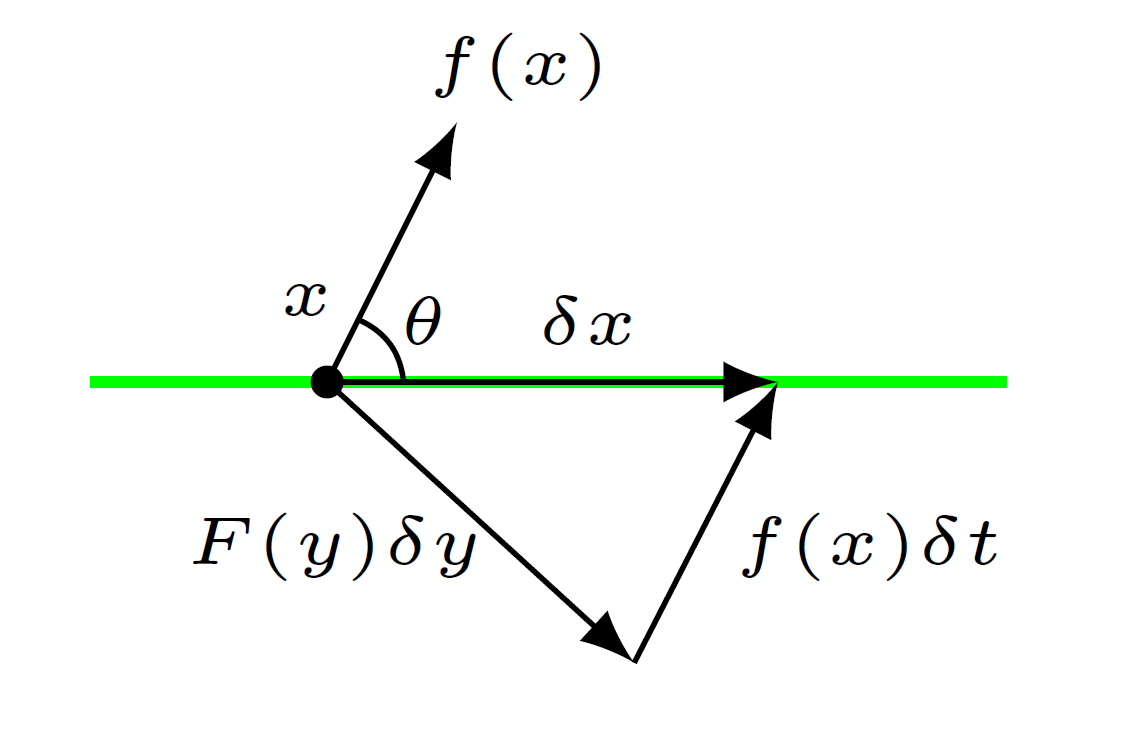}
		\caption{The vector $\delta x= F(y)\delta y + f(x)\delta t$, where the green line is the tangent to $S$ at the point $x$.}
	\end{figure}
	Let $V(u,v)$ be the area of the parallelogram spanned by the two vectors $u$ and $v$. Additionally, let $\Lambda=\det(F(y))$. Then, we have
	\begin{equation}
	\begin{split}
	V(f(y),\delta y) &= \lVert f(y)\rVert\cdot\lVert \delta y \rVert \sin\alpha\\
	V(\underbrace{F(y)\cdot f(y)}_{=f(x)},F(y)\cdot \delta y) &= \Lambda  \lVert f(y)\rVert\cdot\lVert \delta y \rVert \sin\alpha\\
	&= V(f(x),F(y)\cdot\delta y + f(x)\cdot\delta t)\\
	&= \lVert f(x) \rVert \cdot \lVert \delta x \rVert \sin\theta.
	\end{split}
	\end{equation}
	Collecting terms, we see that 
	\begin{equation}
	\frac{\lVert \delta x \rVert}{\lVert \delta y \rVert} = 
	\frac{\lVert f(y)\rVert}{\lVert f(x)\rVert} 
	\frac{\sin\alpha}{\sin\theta} \Lambda.
	\end{equation}
	Combining this with equation \eqref{eq:Fy}, we arrive at equation \eqref{eq:finally}.
%
\end{proof}
\begin{corollary}
	Suppose now that we have a hybrid periodic orbit that intersects $S$ $n$ times. Let $x_1,\ldots,x_n \in S$ and $y_i=\Delta(x_i)$. Additionally, let $\gamma_i$ be the flow that connects $y_i$ to $x_{i+1}$, i.e. $\gamma_i(0)=y_i$ and $\gamma_i(T_i)=x_{i+1}$. Also, let $\alpha_i$ be the angle $f(y_i)$ makes with $\Delta(S)$ and $\theta_i$ be the angle $f(x_i)$ makes with $S$. Then, the derivative of the Poincar\'e map is given by
	\begin{equation}
	(P^n)'(x_1) = \prod_{i=1}^n \, \Delta'(x_i) \frac{\lVert f(y_i)\rVert}{\lVert f(x_i)\rVert} \frac{\sin\alpha_i}{\sin\theta_i} \, \exp \left( 
	\int_0^{T_i}\, \nabla \cdot f(\gamma_i(t))\, dt \right).
	\end{equation}
\end{corollary}
This gives a precise test for determining the stability of planar hybrid orbits. We would like to extend this to higher dimensions, but we can only calculate $\det P'(x_0)$ and not its individual eigenvalues.

\begin{theorem}
	Assume that $\X=\mathbb{R}^n$ and that $\gamma(\cdot)$ is a periodic orbit intersecting $S$ once with $x\in S$ and $y=\Delta(x)$ and period length $T$. Let $\alpha$ and $\theta$ be described as in Theorem \ref{th:stability}. If $\gamma$ is stable, then
	\begin{equation}\label{eq:determinant}
	\left|\det \left( \Delta'(x) \right) 
	\frac{\lVert f(y)\rVert}{\lVert f(x)\rVert} \frac{\sin\alpha}{\sin\theta} \cdot 
	\exp\left( \int_0^T \, \nabla\cdot f(\gamma(t)) \, dt\right)\right|\leq 1.
	\end{equation}
\end{theorem}
\begin{proof}
	Equation \eqref{eq:determinant} is equal to $\det P'(x)$. Thus, if the determinant is greater than 1, it must have an eigenvalue greater than 1 and the system is unstable.
\end{proof}
\begin{corollary}
	If the expression in \eqref{eq:determinant} has value is less than 1 and the orbit, $\gamma(t)$, is unstable, the point $x_0$ under $P$ must be a saddle type instability.
\end{corollary}
\subsection{Example: Hybrid Van der Pol}
Consider the Van der Pol system
\begin{equation}
\begin{split}
\dot{x} &= y\\
\dot{y} &= \mu (1-x^2)y - x
\end{split}
\end{equation}
If we let $z=[x;y]$ and let $f$ be such that that the dynamics 
is given by $\dot{z}=f(z)$, then $\nabla\cdot f(z) = \mu(1-x^2)$. This allows us to cut up the state space as $P=\{ -1<x<1\}$ and $N=\{ -\infty < x < -1\} \cup \{ 1<x<\infty\}$. The divergence of $f$ is strictly negative on $N$ and strictly positive on $P$. Additionally, it is known that the stable limit cycle of this system intersects both $P$ and $N$; as is required by Dulac's criterion (see for instance \cite{strogatz2014nonlinear}, p. 204). As such, let us take $S=\{(x,y)\in\mathbb{R}^{2}|x=1\}$ because we know the continuous limit cycle intersects $S$. 
\subsubsection{Numerical Simulation}\label{sub:numerical}
Let $\mu=1$ and $\Delta(x,y) = (x,-1.5y)$. Let $z_0=[1;3]$.
\begin{figure}[h!]
	\includegraphics[scale=0.25]{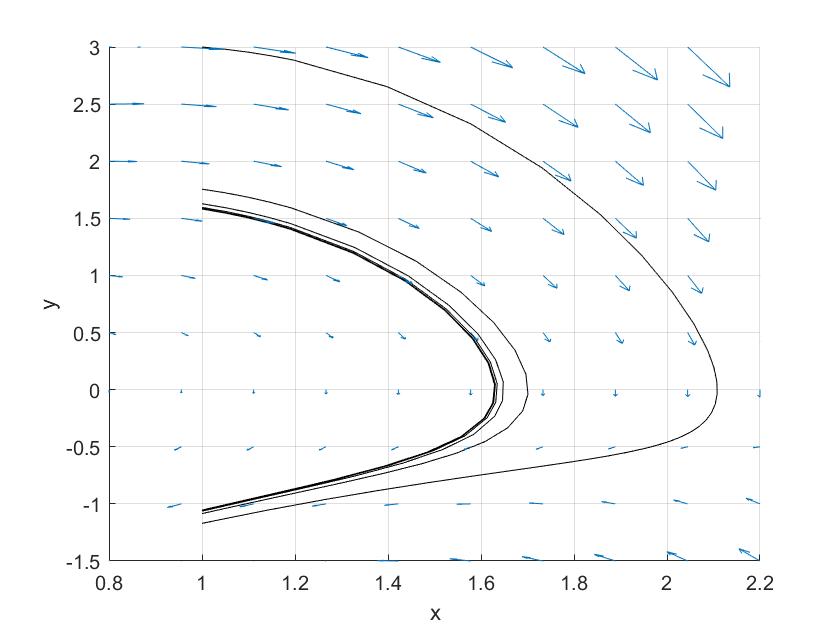}
	\caption{1000 cycles of the flow from \S\ref{sub:numerical}.}
\end{figure}
After running 100 cycles and seeing that the flow ends up being periodic, the initial and final $y$ values are:
\begin{equation}\label{eq:limit_ys}
\begin{array}{rr}
y^- = & -1.0498\\
y^+ = & 1.5747
\end{array}
\end{equation}
Now, we want to calculate the stability of this orbit. We use the following formula for the derivative of the Poincar\'e map:
\begin{equation}\label{eq:deriv}
P'(z) = \Delta'(y^-)\cdot\frac{\lVert f(y^+)\rVert}{\lVert f(y^-)\rVert}
\frac{\sin\alpha}{\sin\theta} \cdot \exp \left(
\int_0^T \, \nabla\cdot f(\gamma(t))\, dt\right)
\end{equation}
This can be interpreted as multiplying together the discrete part, the geometric part and the continuous part.
Numerically integrating over the limit cycle yields a derivative of
\begin{equation*}
|P'|=0.3338
\end{equation*}
\subsubsection{Testing Instability}\label{sub:inst}
Now, we will modify the impact map (while keeping the continuous flow and the impact surface fixed) to make the orbit unstable. We will do this by making the impact map be $\Delta(1,y)=(1,m(y-A)+B)$ where $A=y^-$ and $B=y^+$ as in equation \eqref{eq:limit_ys}. This allows us to control the derivative of $\Delta$ (that is, $m$) while keeping the orbit from changing. Using the results from equation \eqref{eq:deriv}, we see that the derivative is now
\begin{equation}
|P'(y^-)| = 0.2225|m|
\end{equation}
If we run the simulations for $m$ increasing in magnitude past $\approx4.4943$ the orbit should become unstable. Also, the sign of $m$ will determine the number of times the orbit intersects the impact surface.
$$\begin{array}{l|ccccc}
m   & -4.6 & -4.55 & -4.5 & -4.45 & -4.4 \\
\hline
y^+ & 1.6034  & 1.5898  & 1.5768  & 1.5747  & 1.5747 
\end{array}$$
If we let $m$ be positive, the resulting unstable periodic orbit will intersect the impact surface twice.
$$\begin{array}{l|ccccc}
m     & 4.4 & 4.45 & 4.5 & 4.55 & 4.6 \\
\hline
y^+_1 & 1.5747 & 1.5747 & 1.5059 & 1.3758 & 1.3091 \\
y^+_2 & 1.5747 & 1.5747 & 1.6475 & 1.8119 & 1.9132
\end{array}$$
\begin{figure}[h!]
	\includegraphics[scale=0.35]{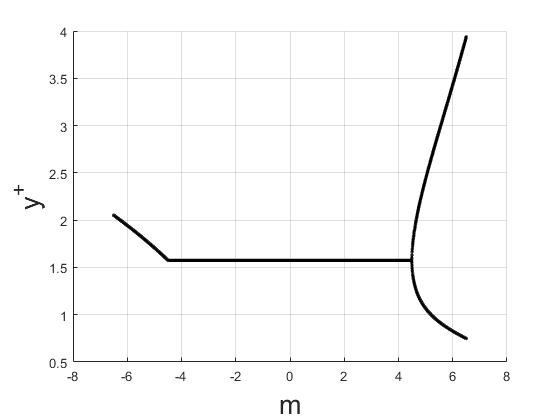}
	\caption{Displaying the locations of the jumps after performing 1000 iterations of the system in \S\ref{sub:inst}.}
\end{figure}
All of the numerics were performed with Matlab's ode45 differential equation solver, equation \eqref{eq:deriv} was integrated via the trapezoidal rule, and all tests ran for a duration of 1000 iterations to locate the steady-state.

\subsection{An Analytic Example}
Consider the continuous dynamics (in polar coordinates) 
\begin{equation}
\begin{split}
\dot{r}& = 1-r\\
\dot{\theta}&=1
\end{split}
\end{equation}
Under these continuous dynamics, for all points $x_0=(r_0,\theta_0)$, $\omega_c(x_0) = S^1$. Additionally, the flow of the system is
\begin{equation}\label{eq:flow}
\varphi_t(r_0,\theta_0) = \left( (r_0-1)e^{-t}+1, \theta_0 + t\right).
\end{equation}
The last notable feature is that the divergence is everywhere equal to -1, i.e. $\nabla\cdot f(r,\theta) \equiv -1$.
Let us consider the hybrid system where the impact map is the ray from the origin at angle $\alpha$, that is $S=\{(r,\theta)|\theta=\alpha\}$ and the impact map is given as
\begin{equation}\label{eq:impact}
\Delta(r,\alpha) = (\beta r,\gamma).
\end{equation}
Let us now compute the Poincar\'e map both analytically and by equation \eqref{eq:deriv} to compare. We will assume that $0\leq \gamma < \alpha \leq 2\pi$. Then the time between all impacts is $\alpha-\gamma$. Using the fact that the time between impacts is constant and equations \eqref{eq:flow} and \eqref{eq:impact}, we obtain the Poincar\'e map
\begin{equation}
P(r_0) = \beta \left[ (r_0-1)e^{\gamma-\alpha}+1 \right].
\end{equation}
If $\beta e^{\gamma-\alpha}<1$, this yields a fixed point of
\begin{equation}
r_0^* = \frac{\beta \left[ e^{\gamma-\alpha}-1\right]}{\beta e^{\gamma-\alpha}-1}.
\end{equation}
Then the derivative is
\begin{equation}\label{eq:p_derivative}
P'(r_0^*) = \beta e^{\gamma - \alpha}.
\end{equation}
Now, we compare with equation \eqref{eq:deriv}. Computing each of the three pieces,
\begin{equation}
\begin{split}
\Delta'(y^-) = \beta\\
\frac{\lVert f(y^+)\rVert}{\lVert f(y^-)\rVert}
\frac{\sin\alpha}{\sin\theta} = 1\\
\exp\left(\int_0^T \, \nabla\cdot f(\gamma(t))\, dt\right) = e^{\gamma-\alpha}.
\end{split}
\end{equation}
which matches up with equation \eqref{eq:p_derivative}.

\section{Poincar\'e-Bendixson theorem for 1-dimensional hybrid dynamical systems}
\label{sec:1dim}

In this section we contrast the above
results to a  Poincar\'e-Bendixson theorem for hybrid systems in one dimension.

\begin{lemma}\label{le:1}
	Let $S\subset\mathbb{R}$ be a finite set and $P:S\rightarrow S$. Then, for all $x\in S$ there exists $N\ne M$ large enough such that $g^N(x)=g^M(x)$.
	Specifically, $\omega_d(x)$ is a periodic orbit.
\end{lemma}
\begin{proof}
	Fix a $x\in S$. Define the sequence $\{x_n\}_{n\in\mathbb{N}}$ where $x_n = P^n(x)$. Since the set $S$ is compact, by Bolzano-–Weierstrass there exists a convergent subsequence, $\{x_{n_k}\}_{k\in\mathbb{N}}{\subset \{x_n\}_{n\in\mathbb{N}}}$. Call the limit $\overline{x}$. Since $S$ is uniformly separated, there exists a $K$ large enough such that for all $p\geq K$, $x_{n_p}=\overline{x}$. Take $N=n_K$ and $M=n_{K+1}$ and we have found our periodic orbit.
\end{proof}

Here, we prove a version of Poincar\'e-Bendixson for a much more general class of hybrid systems in one dimension. In this section, we drop the assumptions (H.1)-(H.6) and replace them with the following:
\begin{enumerate}
	\item[(A.1)] $\X\subset\mathbb{R}$ is open and connected.
	\item[(A.2)] $f:\X\rightarrow\mathbb{R}$ is $C^1$.
	\item[(A.3)] $S$ is a subset of $\mathbb{R}$.
	\item[(A.4)] $\Delta:S\rightarrow\mathbb{R}$.
\end{enumerate}
Under these considerably weaker assumptions (which requires the dimension restriction) we can prove the following theorem. Recall {$\N{\varepsilon}$ from} equation \eqref{eq:fixed_sets}.
\begin{theorem}\label{th:main}
	If either
	\begin{enumerate}
		\item[(S.1)] $S\subset\mathbb{R}$ is uniformly discrete, that is 
		$$\inf_{\substack{x,y\in S\\ x\ne y}} \! |x-y|=\delta>0.$$
		\item[(S.2)] The image of $\Delta$ is far from $S$ if we are away from a fixed point of $f$, that is for $\varepsilon>0$
		$$\inf_{x,y\in S\setminus \N{\varepsilon}} |\Delta(x)-y| = \eta(\varepsilon) >0.$$
	\end{enumerate}
	Then, if $R\subset\mathbb{R}$ is a forward invariant, compact set and for some $x\in R$ such that ${\omega_H(x)}\cap\fix(f)=\emptyset$, then $\omega_H(x)$ is a limit-cycle. Moreover, $\omega_H(x)\subset \overline{o_H^+(x)}$.
\end{theorem}
First note that condition (S.1) is similar to (H.4) and condition (S.2) is similar to (H.6). Before we can prove this result, we need to go through some {preliminaries results given in the following lemmas.}
Do note, however, that Proposition \ref{th:closed} still holds for {this class of} HDSs, i.e. $\omega_H(x)$ is still a closed set.
\begin{lemma}\label{le:2}
	Let $R$ be a compact, forward invariant set. Fix $x\in R$. Then for all $\varepsilon>0$ there exists $T>0$ such that for all $t>T$, $\varphi_t^H(x)\in \mathcal{B}_\varepsilon(\omega_H(x))$.
\end{lemma}
\begin{proof}
	Fix $\varepsilon>0$. Suppose the for all $T>0$, there exists $t>T$ such that $\varphi_t^H(x)\not\in\mathcal{B}_\varepsilon(\omega_H(x))$. So let $T_n\rightarrow\infty$ and choose $t_n>T_n$ such that $\varphi_{t_n}^H(x)\not\in\mathcal{B}_\varepsilon(\omega_H(x))$. Then, the sequence $\{ \varphi_{t_n}^H(x)\}_{n\in\mathbb{N}}$ is far away from $\omega_H(x)$. But, because $R$ is compact, by Bolzano--Weierstrass, there exists a convergent subsequence, $\varphi_{t_{n_k}}^H(x)\rightarrow\overline{x}$. By the definition of $\omega_H(x)$, $\overline{x}\in\omega_H(x)$.
\end{proof}
\begin{lemma}\label{le:3}
	If $\fix(f)\cap{\omega_H(x)} = \emptyset$ and $x\in R$ as in Theorem \ref{th:main}, then 
	$$\text{dist} \left( o_H^+, \fix(f)\right) = \delta > 0.$$
	i.e. $o_H^+(x)\cap \N{\delta}=\emptyset$.
\end{lemma}
\begin{proof}
	Because $f$ is {a $C^1$ function and therefore} a Lipschitz function, $\fix(f)$ is a closed set. Since $\omega_H(x)\subset R$, ${\omega_H(x)}$ is compact. This implies that since $\fix(f)$ and $\omega_H(x)$ are disjoint, they are uniformly separated. So there exists an $\varepsilon>0$ such that $\N{\varepsilon}\cap \omega_H(x)=\emptyset$. By Lemma \ref{le:2} for $T>0$ large enough, all $t>T$ we have $\varphi_t^H(x)\in \mathcal{B}_{\varepsilon/2}(\omega_H(x))$. This tells us that for sufficiently large times, the forward orbit of $x$ is far away form $\fix(f)$. So, we just need to examine the orbit up to time $T$. Call the set $o_H^T(x) = \{ \varphi_t^H(x) : t\in [0,T]\}$. We know that $o_H^T(x)$ is disjoint from $\fix(f)$, but because $o_H^T(x)$ is not closed, we can't say for sure that it is uniformly distant. The only points that can cause trouble are the points close to $o_H^T(x)$ but not in the set. The only points that fit this bill are the break points. However, if one of the break points of the flow is a fixed point of $f$, the flow would approach it asymptotically and thus the limit set would be that point. 
\end{proof}
It is interesting to note that because both Lemmas \ref{le:2} and \ref{le:3} do not require assumption (A.1), they still hold for HDS as defined by definition \ref{def:smooth_hybrid}.
\begin{lemma}\label{le:4}
	Let $f$, $S$, $\Delta$, $R$, and $x$ be as in Theorem \ref{th:main}. Then, for all $y\in o_H^+(x)$ there exists a time, $t_0$, such that $\varphi_{t_0}(x)\in S$.
\end{lemma}
\begin{proof}
	Assume not. Then, $o_H^+(y)$ never jumps. So we can replace it with $o^+(y)$.
	But, by Lemma \ref{le:2}, $o^+(y)$ is uniformly far from $\fix(f)$. So the flow of $y$ is either monotonically increasing or decreasing for all time with a speed bounded away form zero. This means that $y$ must approach either $+\infty$ or $-\infty$ as time approaches infinity. This contradicts the assumption that $o^+(y)$ is confined to a compact set. 
\end{proof}

\begin{proof}[Proof of Theorem \ref{th:main}]
	First, let us assume that condition (S.1) holds. Then there exist finitely many points inside $R\cap S$. Label these points in ascending order $s_1,\ldots,s_n$. Define the set $E:=\{s\in R\cap S | \Delta^n (s) \in R\cap S, \forall n\}$. Then, since $E$ is a finite set with discrete dynamics by Lemma \ref{le:1}, $x\in E$ must eventually be a fixed point or a periodic orbit. So, if $x\in E$ then $\omega_H(x)$ is a periodic orbit. Additionally, if there exists any time where the orbit of $x$ intersects $E$, then $\omega_H(x)$ is a periodic orbit. So, let's assume that $o_H^+(x)\cap E=\emptyset$.\\
	Without loss of generality, let $x_0\not\in S$. 


Then, by Lemma \ref{le:4}, there exists a point $s_{k_0}\in R\cap S$ such that the flow $\varphi_{t_0}(x) = s_{k_0}$. Now, let $x_1 := \Delta(s_{k_0})$ and let $s_{k_1}$ be the impact point $x_1$ gets mapped to. This gives dynamics on the impact points,
	$$s_{k_{n+1}} = \mathcal{M}(s_{k_n}).$$


	But since there are only finite many $s_j$'s, we must either end up with a periodic orbit or a fixed point (Lemma \ref{le:1}). Thus $\omega_H(x)$ is a limit cycle.\\
	
	Now, assume condition (S.2) holds. Since ${\omega_H(x)}$ contains no fixed points, $o_H^+(x)$ is uniformly far from roots of $f$ (Lemma \ref{le:3}). Let us rename the set $R$ to be $R=\overline{o_H^+(x)}$ (which is closed). Then, 
	$$R\cap S = (R\cap S)\setminus \N{\delta} $$ 
	with $\N{\varepsilon}$ as in equation \eqref{eq:fixed_sets}.

	So, condition (2) tells us that there exists some positive $\eta$ such that
	$$\inf_{x,y\in R\cap S} \! d(\Delta(x),y) = \eta >0.$$
	Additionally, let $\xi := \displaystyle \inf_{x\in R} |f(x)|$. Then, the minimal time between consecutive impacts is bounded below by $\eta/\xi$. By concatenating the smooth dynamics between impacts and using Lemma \ref{le:4}, the orbit looks like
	$$o_H^+(x) = \bigsqcup_{n=0}^\infty \; [a_{k_n},b_{k_{n+1}}).$$
	But each interval has a minimal length of $\eta$ and therefore since $R$ is compact, they must eventually intersect. Additionally, if $[a_j,b_j)\cap[a_i,b_i)\ne\emptyset$ then $b_j=b_i$. This shows that there only may exist finitely many distinct $b_{k_n}$'s. This allows us to define dynamics on a finite set,
	$$b_{k_{n+1}} = \mathcal{M}(b_{k_n}).$$
	and thus, a periodic orbit of $\mathcal{M}$ must exist (Lemma \ref{le:1}). Therefore $\omega_H^+(x)$ is a limit-cycle.\\
	
	Since we construct a periodic orbit via Lemma \ref{le:1} for both cases (S.1) and (S.2), we point out that Lemma \ref{le:1} states that we hit the periodic orbit after finitely many impacts. Thus, the forward orbit enters $\omega_H(x)$ at some finite time and $\omega_H(x)\subset \overline{o_H^+(x)}$.
\end{proof}


\section{Application to periodic walking: the rimless wheel}\label{sec:examples}

The rimless wheel is a one-degree-of-freedom hybrid mechanical system in which the guard is reached when the swinging spoke makes contact with the inclined plane (see Figure \ref{fig:rw_schematic}. For a rimless wheel rolling along an inclined plane an analytically computable stable limit cycle exits \cite{Coleman1998}.

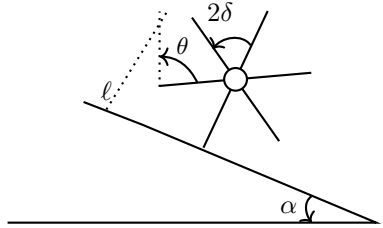
\begin{figure}[h!]
\centering
  \begin{tikzpicture}[scale=1]
    \draw[thick] (1,0.8) to (1.77,0.5) -- (4.85,-0.8)--(0,-0.8);
    \draw[thick] (3,1.1) circle (.15cm);
    \spoke{(3.0,1.1)}{-85};
    \spoke{(3.0,1.1)}{-25};
    \spoke{(3.0,1.1)}{35};
    \spoke{(3,1.1)}{95};
    \spoke{(3,1.1)}{155};
    \spoke{(3,1.1)}{-145};
  \draw[dotted,thick] (2,1) -- (2,2);
    \draw[->,thick] (4,-.45) to  [out=-145, in =125] (4,-0.8);
    \node at (3.7,-0.6) {$\alpha$};
    \draw[->, thick] (3.2,1.55) to [out=125, in =55] (2.7,1.5);
    \node at (2.8,2) {$2\delta$};
    \draw[->,thick] (2.5,1.05) to [out=115, in =0] node[pos=0.5, above]{\small $\theta$} (2,1.4) ;
    \draw[dotted, thick] (1.3,.7)  -- node [pos=0,above] {$\ell$} (2.1,2);
  \end{tikzpicture}
  \caption{The rimless wheel.}
  \label{fig:rw_schematic}
\end{figure}

For this system let $x=(\theta,\dot{\theta})$, the continuous dynamics are given by equations \eqref{eq:wheel_cont} and \eqref{eq:wheel_disc} below (see {\cite{Coleman1998} and } \cite{saglam2014lyapunov} for an in depth formulation of this problem). {We  assume  the  mass $m$ is  lumped  into  the center  of  the  robot,  the  length  of  each  leg  is  given  by $\ell$, and each  inter-leg  angle is $2\delta= \frac{2\pi}{N}$, with $N$ being the number of legs}. Here, $\delta$ is the angle the leg makes with the ground when it lifts off and $\alpha$ is the grade of the slope the passive walker is walking down.
\begin{equation}\label{eq:wheel_cont}
\dot{x} = f(x) = \left[\begin{array}{c} x_2 \\ \zeta \sin(x_1) \end{array}\right],  \quad\zeta = g/\ell
\end{equation}
The impact surface is given by $S = \{x_1 = -\delta-\alpha\}$ and the impact map is
\begin{equation}\label{eq:wheel_disc}
\Delta(x) = \left[ \begin{array}{c}
\delta-\alpha \\
\cos(2\delta)x_2 \end{array}\right].
\end{equation}

By applying Theorem \ref{th:HPB}, all the smooth hybrid assumptions (H.1)-(H.6) are satisfied as well as points (C.1), (C.2), and (C.4). The transversality condition, (C.5), is satisfied as long as the trajectory stays away from the origin. To find the forward invariant compact set free of fixed points we do an energy balance. 

The (potential) energy gained over a single swing is
\begin{equation}\label{eq:potential}
\Delta P = 2\ell g\sin\delta\sin\alpha.
\end{equation}
While the amount of (kinetic) energy lost at impact is
\begin{equation}\label{eq:kinetic}
\Delta V = \frac{1}{2}(\ell x_2^-)^2\left( 1-\cos^2 2\delta \right).
\end{equation}
Call the total energy of the system $E$. If $\Delta V > \Delta P$, then $E(P(x)) < E(x)$. And if $\Delta V < \Delta P$, then $E(P(x))>E(x)$. This is how we will locate a forward invariant compact (FIC) set.

Clearly, for $x_2^-$ large enough, more kinetic energy will be lost through impacts than is acquired over the swing phase. The remaining question is for $x_2^-$ being small enough that we gain more energy. 

If $\delta > \alpha$, then we can fail to swing forward. In this case, we can calculate the minimum velocity, $x_2$, needed at the beginning of the swing phase to make it to the next one. 
\begin{equation}
(x_2^+)^2 > 2\zeta \left( 1- \cos(\delta-\alpha)\right),\quad
(x_2^-)^2 > \frac{2\zeta\left( 1- \cos(\delta-\alpha)\right)}{\cos^2(2\delta)}
\end{equation}
If we start the swing with this velocity and by equations \eqref{eq:potential} and \eqref{eq:kinetic} we gain energy, then
\begin{equation}
\Delta V = \frac{\ell^2}{2}\left( 1-\cos^22\delta\right) \left(
\frac{2\zeta\left( 1- \cos(\delta-\alpha)\right)}{\cos^2(2\delta)} \right).
\end{equation} Therefore, if $\delta>\alpha$ and 
	\begin{equation}\label{eq:existance_walking}
	2\sin\delta\sin\alpha > 
	\left( 1-\cos^22\delta\right) \left(
	\frac{\left( 1- \cos(\delta-\alpha)\right)}{\cos^2(2\delta)} \right),
	\end{equation}
	then there exists at least one periodic orbit that intersects the impact surface either once or twice.
	
In Figure \ref{fig:stability_regions} (left) we show the parameters $\alpha$ and $\delta$ where a stable limit cycle exists as well as (right) a trajectory and the region of attraction for the  choice set of parameters $\delta=\pi/10$, $\alpha=\pi/30$, and $\zeta=9.8$.

\begin{figure}[h!]
	\centering
	\begin{subfigure}[t]{0.45\textwidth}
		\includegraphics[width=\textwidth]{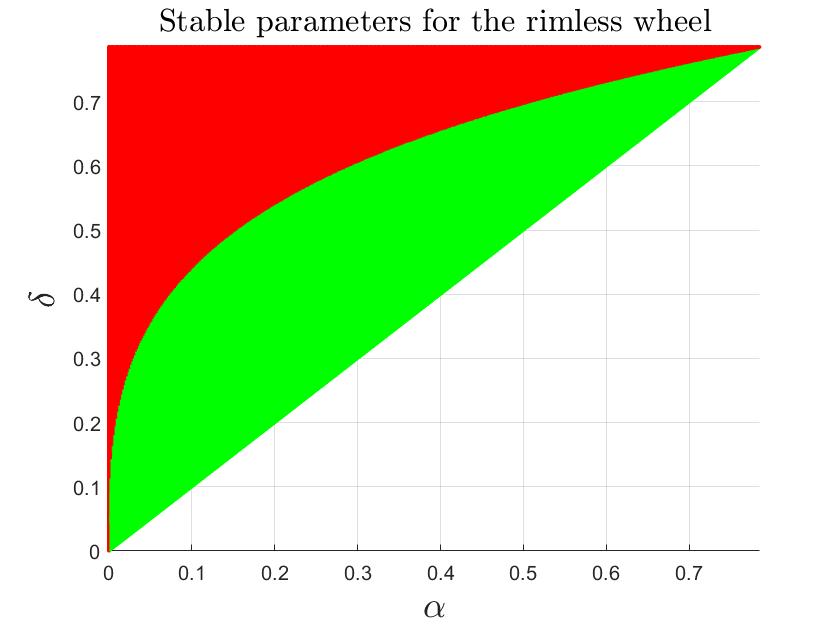}
	\end{subfigure}
	\quad
	\begin{subfigure}[t]{0.45\textwidth}
		\includegraphics[width=\textwidth]{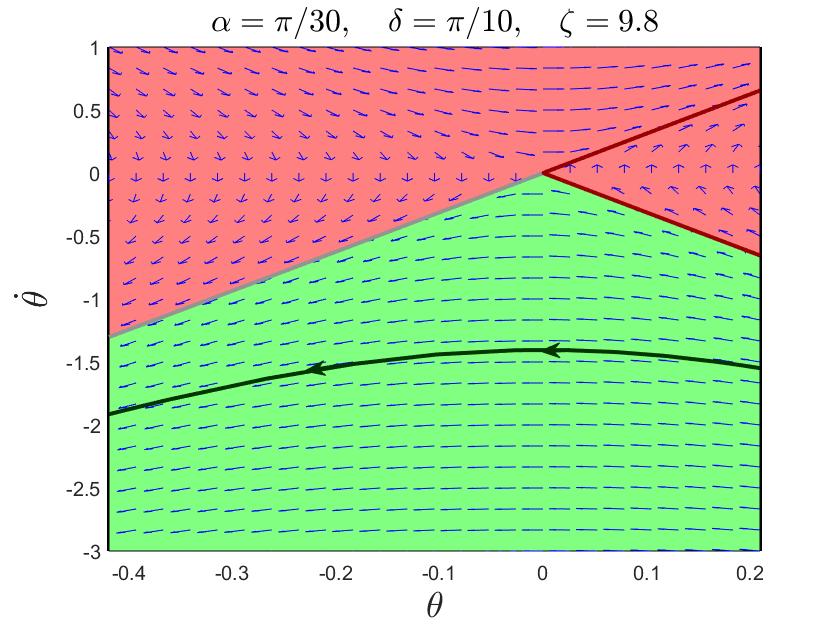}
	\end{subfigure}
\caption{Left: The green region indicates values of $\alpha$ and $\delta$ where there exists a limit cycle as predicted by equation \eqref{eq:existance_walking}. Right: The green region is the domain of attraction for the limit cycle, whose existence is guaranteed by equation \eqref{eq:existance_walking}.}
\label{fig:stability_regions}
\end{figure}


%
\appendix
\section{Continuous-time Poincar\'e-Bendixson theorem and stability of periodic orbits}\label{appendix}
Because HDS are a mixture of continuous and discrete dynamics {in this appendix we review the main results used in the work for} both continuous and discrete {dynamical} systems.

For a (smooth) manifold $\X$, let $f$ be a $C^1$ vector field on $\X$. Then we can study the flow induced by the vector field, 
\begin{equation}\label{eq:continuous}
\varphi_t:\X\rightarrow\X,\quad \frac{d}{dt}\varphi_t(x) = f(\varphi_t(x)).
\end{equation}
Then, assuming the flow is complete (e.g. p. 62 in \cite{bloch2008nonholonomic}), let us define the (continuous) forward orbit and $\omega_c$-limit set. 
\begin{equation}
o_c^+(x) := \left\{ \varphi_t(x) : t\in\mathbb{R}^+\right\}.
\end{equation}
\begin{equation}
\omega_c(x) := \left\{ y\in\X : \exists t_n\rightarrow\infty~s.t. \lim_{n\rightarrow\infty} \varphi_{t_n}(x)=y \right\}.
\end{equation}
\begin{theorem}[\cite{perko1991differential}, p. 193]\label{th:perko}
	Assume that $\X\subset\mathbb{R}^n$. The set $\omega_c(x)$ is a closed subset of $\mathbb{R}^n$. Furthermore, if $o_c^+(x)$ is contained in a compact subset of $\mathbb{R}^n$, then $\omega_c(x)$ is a non-empty, connected and compact subset of $\mathbb{R}^n$.
\end{theorem}
Now, we can state the Poincar\'e-Bendixson Theorem for $\mathbb{R}^2$.
\begin{theorem}[Poincar\'e-Bendixson, \cite{perko1991differential}, p. 245]\label{th:PB}
Suppose that $f\in C^1(\X)$, where $\X$ is an open subset of $\mathbb{R}^2$, and that $o_c^+(x)$ is contained in a compact subset $F$ of $\X$. Then, if $\omega_c(x)$ contains no fixed points of $f$, $\omega_c(x)$ is a periodic orbit.
\end{theorem}
While Theorem \ref{th:PB} can assert the existence of periodic orbits, it says little about their stability. A {method} for  studying this behavior is the 
theory of  Poincar\'e maps. Let $\gamma$ be a periodic orbit of $\varphi_t$ and choose a hyper-surface, $\Sigma$ that is transverse to the flow at a point $x_0=\gamma(t_0)$ (that is $f(x_0)\cdot n(x_0)\ne0$ where $n(x_0)$ is the unit normal vector to $\Sigma$ at $x_0\in \Sigma$). Then, for points $y$ near $x_0$ in $\Sigma$, define $\tau(y)$ to be the time for the flow starting at $y$ to return to $\Sigma$. Then the Poincar\'e map (or first return map), is defined as
$$P(y)=\varphi_{\tau(y)}(y).$$
See \cite{guckenheimer2002nonlinear} and \cite{perko1991differential} for more information on Poincar\'e maps. If we can determine the derivative of $P$, then an orbit is stable if all the eigenvalues of $P'$ have modulus less than one (see \cite{bookHSD}, p. 219). In general it is not possible to analytically compute $P'$, but there are a helpful results in the literature (\cite{strogatz2014nonlinear}, p. 282). The next two results deal with the differentiability of $P$ as well as a way to compute $P'$ for planar systems.

\begin{theorem}[\cite{perko1991differential}, p. 212]
	Let $\X$ be an open subset of $\mathbb{R}^n$ and let $f\in C^1(\X)$. Suppose that $\varphi_t(x_0)$ is a periodic solution of \eqref{eq:continuous} of period $T$ and that the cycle 
	$$\Gamma=\left\{ x\in \mathbb{R}^n : x=\varphi_t(x), 0\leq t\leq T\right\}$$
	is contained in $\X$. Let $\Sigma$ be the hyperplane orthogonal to $\Gamma$ at $x_0$; i.e. let 
	$$\Sigma = \left\{ x\in\mathbb{R}^n : (x-x_0)\cdot f(x_0)=0\right\}.$$
	Then there is a $\delta>0$ and a unique function $\tau(x)$ defined and continuously differentiable for $x\in \mathcal{B}_\delta(x_0)$, such that $\tau(x_0)=T$ and
	$$\varphi_{\tau(x)}(x)\in\Sigma$$
	for all $x\in \mathcal{B}_\delta(x_0)$.
\end{theorem}
\begin{theorem}[\cite{perko1991differential}, p. 216]\label{th:AA}
	Let $\X$ be an open subset of $\mathbb{R}^2$ and suppose that $f\in C^1(\X)$. Let $\gamma(t)$ be a periodic solution of \eqref{eq:continuous} of period $T$. Then the derivative of the Poincar\'e map $P(s)$ along a straight line $\Sigma$ normal to $\Gamma=\{x\in\mathbb{R}^2|x=\gamma(t)-\gamma(0),~t\in[0,T]\}$ at $x=0$ is given by
	\begin{equation}
	P'(0)=\exp \left(\int_0^T \, \nabla\cdot f(\gamma(t))\, dt\right).
	\end{equation}
\end{theorem}
This gives us a straightforward method for finding the stability of a periodic orbit for planar systems. The problem is that when going to higher dimensions, we can no longer systematically find a derivative of $P$. However, we can still get necessary conditions for stability.
\begin{theorem}[\cite{perko1991differential}, p. 230]
	Let $f\in C^1(\X)$ where $\X$ is an open subset of $\mathbb{R}^n$ containing a periodic orbit $\gamma(t)$ of \eqref{eq:continuous} of period $T$. Then, $\gamma(t)$ is not asymptotically stable unless
	\begin{equation}
	\int_0^T\, \nabla\cdot f(\gamma(t))\, dt \leq 0.
	\end{equation}
\end{theorem}


\end{document}